\documentclass{article}
\usepackage[british]{babel}
\usepackage{amsthm}
\usepackage{amssymb, enumerate}
\usepackage{latexsym}
\usepackage{url}
\usepackage[all,cmtip]{xy}
\usepackage{amsmath,amssymb}

\newcounter{itemcounter}
\numberwithin{itemcounter}{section}

\newtheorem{thm}[itemcounter]{Theorem}
\newtheorem{lem}[itemcounter]{Lemma}
\newtheorem{defi}[itemcounter]{Definition}
\newtheorem{prop}[itemcounter]{Proposition}
\newtheorem{cor}[itemcounter]{Corollary}

\newtheorem{rem}[itemcounter]{Remark}

\newtheorem*{thm*}{Theorem}
\newtheorem*{con*}{Conjecture}
\newtheorem*{cor*}{Corollary}
\newtheorem*{ack*}{Acknowledgements}

\newcommand{\Inf}{\mathop{\rm Inf}\nolimits}

\newcommand{\Irr}{\mathop{\rm Irr}\nolimits}
\newcommand{\IBr}{\mathop{\rm IBr}\nolimits}

\newcommand{\Hom}{\mathop{\rm Hom}\nolimits}

\newcommand{\Aut}{\mathop{\rm Aut}\nolimits}

\newcommand{\Stab}{\mathop{\rm Stab}\nolimits}
\newcommand{\Br}{\mathop{\rm Br}\nolimits}

\newcommand{\Pic}{\mathop{\rm Pic}\nolimits}
\newcommand{\rk}{\mathop{\rm rk}\nolimits}

\newcommand{\nth}{\mathop{\rm th}\nolimits}
\newcommand{\nd}{\mathop{\rm nd}\nolimits}
\newcommand{\Id}{\mathop{\rm Id}\nolimits}

\newcommand{\CF}{\mathop{\rm CF}\nolimits}
\newcommand{\prj}{\mathop{\rm prj}\nolimits}

\newcommand{\GL}{\mathop{\rm GL}\nolimits}

\newcommand{\cT} {\mathcal{T}}
\newcommand{\cL} {\mathcal{L}}
\newcommand{\cE} {\mathcal{E}}
\newcommand{\cO} {\mathcal{O}}
\newcommand{\cJ} {\mathcal{J}}

\newcommand{\cB} {\mathcal{B}}

\title{On Picard groups of blocks with normal defect groups  \footnote{This research was supported by the EPSRC (grant nos. EP/M015548/1 and EP/T004606/1).} }
\author{Michael Livesey\footnote{School of Mathematics, University of Manchester, Manchester, M13 9PL, United Kingdom. Email: michael.livesey@manchester.ac.uk}}
\date{}

\begin{document}

\maketitle

\begin{abstract}
Let $b$ be a block with normal abelian defect group and abelian inertial quotient. We prove that every Morita auto-equivalence of $b$ has linear source. We note that this improves upon results of Zhou and also Boltje, Kessar and Linckelmann.
\end{abstract}

\section{Introduction}

Let $p$ be a prime, $(K,\cO,k)$ a $p$-modular system with $k$ algebraically closed and $b$ a block of $\cO H$, for a finite group $H$. We always assume that $K$ contains all $|H|^{\nth}$ roots of unity. The Picard group $\Pic(b)$ of $b$ consists of isomorphism classes of $b$-$b$-bimodules which induce $\cO$-linear Morita auto-equivalences of $b$. For $b$-$b$-bimodules $M$ and $N$, the group multiplication is given by $M\otimes_bN$. $\cT(b)$ (respectively $\cL(b)$, $\cE(b)$) will denote the subgroup $\Pic(b)$ of $\cO(H\times H)$-modules with trivial (respectively linear, endopermutation) source.
\newline
\newline
There are many open problems concerning Picard groups. It is proved in~\cite[Corollary 1.2]{ei19} that $\Pic(b)$ is finite. However, it is yet to be proved that $\Pic(b)$ is always bounded in terms of a function of the defect group. There are also no known examples of a block $b$ with $\Pic(b)\neq\cE(b)$. Our main result (see Theorem~\ref{thm:main}) is as follows:

\begin{thm*}
Let $b$ be a block with a normal abelian defect group and abelian inertial quotient, then $\Pic(b)=\cL(b)$.
\end{thm*}

We note that this improves upon a result of Zhou~\cite[Theorem 14]{zh05}. Zhou proves that if $b=\cO(D\rtimes E)$, where $D$ is an abelian $p$-group and $E$ is an abelian $p'$-group of $\Aut(D)$, then $\Pic(b)=\cL(b)$. We can also compare with a result of Boltje, Linckelmann and Kessar~\cite[Proposition 4.3]{bkl18}, where it is assumed in addition that $[D,E]=D$ but the result is that $\Pic(b)=\cT(b)$. Note that this result follows immediately from Corollary~\ref{cor:main}.
%

The following notation will hold throughout this article. We define $\overline{\phantom{A}}:\cO\to k$ to be the natural quotient map and, for a finite group $H$, we extend this to the corresponding ring homomorphism $\overline{\phantom{A}}:\cO H\to kH$. A block $b$ of $H$ will always mean a block of $\cO H$. We set $\Irr(H)$ (respectively $\IBr(H)$) to be the set of ordinary irreducible (respectively irreducible Brauer) characters of $H$ and $\Irr(b)\subseteq\Irr(H)$ (respectively $\IBr(b)\subseteq\IBr(H)$) the set of ordinary irreducible (respectively irreducible Brauer) characters lying in the block $b$. If $N\lhd H$ and $\chi\in\Irr(N)$, then we denote by $\Irr(H,\chi)$ the set of irreducible characters of $H$ appearing as constituents of $\chi\uparrow^H$. Similarly, we define $\Irr(b,\chi):=\Irr(b)\cap\Irr(H,\chi)$. $1_H\in\Irr(H)$ will designate the trivial character of $H$. We use $e_b\in\cO H$ to denote the block idempotent of $b$ and $e_\chi\in KH$ to denote the character idempotent associated to $\chi\in\Irr(H)$. Finally, we set $[h_1,h_2]:=h_1^{-1}h_2^{-1}h_1h_2$, for $h_1,h_2\in H$.
\newline
\newline
The article is organised as follows. In $\S$\ref{sec:abab} we establish some preliminaries about abelian $p'$-groups acting on abelian $p$-groups. We introduce a particular block with normal abelian defect group and abelian inertial quotient in $\S$\ref{sec:Bchar}. $\S$\ref{sec:perfisom} is concerned with perfect isometries and how they relate to our main theorem. In $\S$\ref{sec:one} we study the specifc case of a block with one simple module in greater detail and our main theorem is proved in $\S$\ref{sec:main}.

\section{Abelian $p'$-groups acting on abelian $p$-groups}\label{sec:abab}

\begin{defi}\label{def:action}
Let $H$ be a finite abelian $p'$-subgroup of $\Aut(P)$, for some abelian $p$-group $P$. We say $H$ acts on $P$. If there exists a non-trivial direct decomposition $P\cong P_1\times P_2$ such that $P_1$ and $P_2$ are both $H$-invariant, then we say $H$ acts decomposably on $P$. Otherwise we say $H$ acts indecomposably on $P$.
\end{defi}

We note that our definition of $H$ acting on $P$ is non-standard in that we are demanding that $H$ is a finite $p'$-group, $P$ is a finite abelian $p$-group and that $H$ is a subgroup of $\Aut(P)$ (usually called a faithful action) and not just that we have a group homomorphism $H\to\Aut(P)$. Whenever we have $H$ acting on $P$ we always have the semi-direct product $P\rtimes H$, defined through this action, in mind. We will borrow notation from this setup, for example $C_P(H)$ will denote the set of fixed points in $P$ under the action of $H$.

\begin{rem}\label{rem:ind}
By~\cite[$\S3$, Theorem 3.2]{gor80} we need only require $P_1$ to be $H$-invariant in Definition~\ref{def:action}. Note also that~\cite[$\S5$, Theorem 2.2]{gor80} says that if $H$ acts indecomposably on $P$, then $P$ is necessarily homocyclic.
\end{rem}

The following is proved in~\cite[Theorem 2.3]{gor80}.

\begin{lem}\label{lem:decompP}
Let $H$ act on $P$. Then $P=[P,H]\times C_P(H)$.
\end{lem}

\begin{lem}\label{lem:actirr}
Let $H$ act on $P$. The natural action of $H$ on $\Irr(P)$ has a non-trivial fixed point if and only if its action on $P$ does.
\end{lem}

\begin{proof}
By Lemma~\ref{lem:decompP}, $P=[P,H]\times C_P(H)$. Therefore, if $P$ has a non-trivial fixed point we can certainly construct some non-trivial fixed point of $\Irr(P)$. The converse follows since we can identify the action of $H$ on $P$ with that of $H$ on $\Irr(\Irr(P))$.
\end{proof}

We denote by $\Phi(P)$ the Frattini subgroup of $P$.

\begin{lem}\label{lem:fratt}
Let $H$ act indecomposably on $P\cong (C_{p^n})^m$, for some $m,n\in\mathbb{N}$. Then we have an induced action of $H$ on $P/\Phi(P)\cong (C_p)^m$ and this action is also indecomposable.
\end{lem}

\begin{proof}
The fact that we have an induced action follows from~\cite[$\S5$, Theorem 1.4]{gor80}. Assume $H$ acts decomposably on $P/\Phi(P)$. Let $x\in P\backslash\Phi(P)$ be such that $x\Phi(P)$ is contained in a non-trivial $H$-invariant direct factor of $P/\Phi(P)$ and consider the smallest $H$-invariant subgroup $Q$ of $P$ containing $x$. Certainly $\{1\}<Q<P$ and $Q\nleq\Phi(P)$. So, by Remark~\ref{rem:ind}, there exists some $H$-invariant homocyclic direct factor $Q'$ of $Q$ also satisfying $\{1\}<Q'<P$ and $Q'\nleq\Phi(P)$. So $Q'\cong C_{p^n}^{m'}$ for some $1\leq m'<n$. In particular $Q'$ is an $H$-invariant direct factor of $P$. Again by Remark~\ref{rem:ind}, this contradicts the indecomposablity of the action of $H$ on $P$.
\end{proof}

In what follows, $J(kP)$ will denote the Jacobson radical of $kP$.

\begin{lem}\label{lem:jacobi}
Let $H$ act indecomposably on $P\cong (C_{p^n})^m$, for some $m,n\in\mathbb{N}$.
\begin{enumerate}
\item $H$ is cyclic and if $g$ is a generator of $H$ then $g$ has $m$ distinct eigenvalues
\begin{align*}
\{\lambda=\lambda^{p^m},\lambda^p,\dots,\lambda^{p^{m-1}}\},
\end{align*}
as a linear transformation of $k\otimes_{\mathbb{F}_p}P/\Phi(P)$.
\item Any non-trivial $g\Psi(P)\in P/\Psi(P)$ has trivial stabiliser in $H$.
\item The representations of $H$ on $k\otimes_{\mathbb{F}_p}P/\Phi(P)$ and $J(kP)/J^2(kP)$ are isomorphic.
\end{enumerate}
\end{lem}

\begin{proof}
Certainly $g^p-1=(g-1)^p\in J^2(kP)$ for any $g\in P$ and so the natural group homomorphism $P\to P/\Phi(P)$ induces an isomorphism
\begin{align*}
J(kP)/J^2(kP)\to J(k(P/\Phi(P)))/J^2(k(P/\Phi(P))).
\end{align*}
Therefore, since by Lemma~\ref{lem:fratt} we have an indecomposable action of $H$ on $P/\Phi(P)$, we assume for the remainder of the proof that $P$ is elementary abelian.
\begin{enumerate}
\item We identify $P$ with $\mathbb{F}_p^m$ and view $H$ as a subgroup of $\mathbf{G}:=\GL_m(\mathbb{F}_p)$. Let $g$ be an element of maximal order in $H$. We factorise the characteristic polynomial of $g$ into irreducible factors $f_1(X)^{n_1}.\dots.f_s(X)^{n_s}$ in $\mathbb{F}_p[X]$, where $f_i(X)$ and $f_j(X)$ are coprime for $i\neq j$. We first note that
\begin{align*}
\{v\in\mathbb{F}_p^m|f_1(g)v=0\}
\end{align*}
is a non-trivial $H$-invariant subspace of $\mathbb{F}_p^m$. Therefore, since $H$ acts indecomposably, we must have $f_1(g)=0$, in particular $s=1$ and $f_1(X)$ has degree $d:=m/n_1$. It follows that $o(g)|(p^d-1)$ and $d$ is the smallest positive integer satisfying this condition, where $o(g)$ is the order of $g$. Then $C_{\mathbf{G}}(g)\cong\GL_{n_1}(\mathbb{F}_{p^d})$ and $g$ is represented in $C_{\mathbf{G}}(g)$ by the scalar matrix with $\lambda$'s on the diagonal for some $\lambda\in\mathbb{F}_{p^d}$ a root of $f_1(X)$ (see for example~\cite[Proposition 1A]{fosr82}).
\newline
\newline
Certainly $o(h)|o(g)$ for each $h\in H\leq C_{\mathbf{G}}(g)$ and so the characteristic polynomial of $h$ in $C_{\mathbf{G}}(g)\cong\GL_{n_1}(\mathbb{F}_{p^d})$ must factorise into linear factors. Exactly as for $g$ in $\GL_m(\mathbb{F}_p)$, the characteristic polynomial of $h$ in $\GL_{n_1}(\mathbb{F}_{p^d}$, must be the power of an irreducible polynomial. Therefore, $h$ is also a scalar matrix in $C_{\mathbf{G}}(g)$ and, since $o(h)|o(g)$, it must be a power of $g$ proving $H$ is cyclic. In particular, $\mathbb{F}_p^m$ decomposes into the direct sum of $n_1$ $H$-invariant subspaces and so, in fact, $n_1=1$. The eigenvalues of $g$ are now just the roots of $f_1(X)$. Since $f_1(X)$ has degree $m$, this proves the first part of the lemma.
\item Note that any power $h$ of $g$ from part (1) has $1$ as an eigenvalue if and only if $h=1$. In other words, $C_H(x)=\{1\}$ for any $x\in P\backslash\{1\}$.
\item We claim that the representation of $H$ on $k\otimes_{\mathbb{F}_p}P$ is isomorphic to that of $H$ on $J(kP)/J^2(kP)$ via
\begin{align*}
k\otimes_{\mathbb{F}_p}P&\to J(kP)/J^2(kP)\\
x&\mapsto1-x,
\end{align*}
for each $x\in P$. The fact that we have a homomorphism follows from
\begin{align*}
1-x^i=i(1-x)+(1-x)[(1+x+\dots+x^{i-1})-i]\in i(1-x)+J^2(kP),
\end{align*}
for all $x\in P$ and $i\in\mathbb{N}$. Next we note that $\{1-x_1,\dots,1-x_m\}$ forms a basis for $J(kP)/J^2(kP)$, where $P=\langle x_1\rangle\times\dots\times\langle x_m\rangle$ and hence we have both surjectivity and injectivity.
\end{enumerate}
\end{proof}

We continue with the hypotheses of Lemma~\ref{lem:jacobi}. Set
\begin{align*}
J_{\cO}(P):=\left\{\sum_{x\in P}\alpha_xx|\sum_{x\in P}\alpha_x=0\right\}=(1-x_1,\dots,1-x_m)\triangleleft\cO P,
\end{align*}
where $P=\langle x_1\rangle\times\dots\times\langle x_m\rangle$. Set
\begin{align*}
J_{\cO,2}(P):=\{x\in J_{\cO}(P)|\overline{x}\in J^2(kP)\}.
\end{align*}
Since $\{1-x_1,\dots,1-x_m\}$ is a basis for $J(kP)/J^2(kP)$ we have
\begin{align*}
J_{\cO}(P)&\backslash J_{\cO,2}(P):=\\
&\{a_1(1-x_1)+\dots+a_m(1-x_m)\in J_{\cO}(P)|a_i\in(\cO P)^\times\text{ for at least one }i\}.
\end{align*}

\begin{lem}\label{lem:tech}
Let $x\in J_{\cO}(P)$, then $x^{p^n}=py$, for some $y\in J_{\cO}(P)$. If, in addition, $p=2$, $n=1$ and $x\in J_{\cO}(P)\backslash J_{\cO,2}(P)$, then $y\in J_{\cO}(P)\backslash J_{\cO,2}(P)$.
\end{lem}

\begin{proof}
Let
\begin{align*}
x=a_1(1-x_1)+\dots+a_m(1-x_m)\in J_{\cO}(P),
\end{align*}
for some $a_i\in\cO P$. Then
\begin{align*}
&(a_1(1-x_1)+\dots+a_m(1-x_m))^{p^n}\\
\equiv&(a_1(1-x_1))^{p^n}+\dots+(a_m(1-x_m))^{p^n}\mod p J_{\cO,2}(P).
\end{align*}
By calculating in $\mathbb{F}_pP$ we have that
\begin{align*}
(1-x_i)^{p^n-1}\in 1+x_i+x_i^2+\dots+x_i^{p^n-1}+p\mathbb{Z}P
\end{align*}
and therefore that $(1-x_i)^{p^n}\in p(1-x_i)\cO P$, for each $1\leq i\leq m$. The first claim follows.
\newline
\newline
If $p=2$ and $n=1$, then $(1-x_i)^2=2(1-x_i)$ and so
\begin{align*}
&(a_1(1-x_1)+\dots+a_m(1-x_m))^2\\
\equiv& 2a_1^2(1-x_1)+\dots+2a_m^2(1-x_m)\mod 2J_{\cO,2}(P)
\end{align*}
and the second claim follows from the comments preceding the Lemma.
\end{proof}

\section{$\cO(D\rtimes E)e_\varphi$ and its characters}\label{sec:Bchar}

We set the following notation that will hold for the rest of the article. Let $D$ be a finite abelian $p$-group, $E$ a finite $p'$-group and $Z\leq E$ a central, cyclic subgroup such that $L:=E/Z$ is abelian. Let $L$ act on $D$ set $G:=D\rtimes E$ through this action. We study $B:=\cO Ge_\varphi$, where $\varphi$ is a faithful character of $Z$. Since $D\lhd G$, any block idempotent of $\cO G$ is supported on $C_G(D)=D\times Z$. Therefore, $B_\varphi$ is a block of $\cO G$ with defect group $D$. Set $D_1:=[D,E]$ and $D_2:=C_D(E)$. By Lemma~\ref{lem:decompP}, we have $D=D_1\times D_2$.
\newline
\newline
Before we go on to describe the irreducible characters of $B$ we focus on $\Irr(E,\varphi)$.

\begin{lem}\label{lem:tenchar}$ $
\begin{enumerate}
\item If $\chi_1,\chi_2\in\Irr(E,\varphi)$, then there exists $\theta\in\Irr(E,1_Z)$ such that $\chi_1.\theta=\chi_2$.
\item $\varphi$ extends in $[Z(E):Z]$ different ways to $Z(E)$. Moreover, there is a bijection
\begin{align*}
\Irr(Z(E),\varphi)&\to\Irr(E,\varphi)\\
\psi&\mapsto\chi_\psi,
\end{align*}
where $\psi\uparrow^E=\chi_\psi^{\oplus m}$ and $\chi_\psi\downarrow_{Z(E)}=\psi^{\oplus n}$, for some $m,n\in\mathbb{N}$. In particular, if $\chi\in\Irr(E,\varphi)$ and $\theta\in\Irr(E,1_Z)$, then $\chi.\theta=\chi$ if and only if $\theta\in\Irr(E,1_{Z(E)})$.
\end{enumerate}
\end{lem}

\begin{proof}$ $
\begin{enumerate}
\item Let $\chi\in\Irr(E,\varphi)$. Then
\begin{align*}
\chi.(1_Z\uparrow^E)=(\chi\downarrow_Z.1_Z)\uparrow^E=(\varphi^{\oplus\chi(1)}.1_Z)\uparrow^E=(\varphi\uparrow^E)^{\oplus\chi(1)}.
\end{align*}
Since every element of $\Irr(E,\varphi)$ appears as a constituent of $\varphi\uparrow^E$ and $1_Z\uparrow^E$ only has constituents in $\Irr(E,1_Z)$, the claim follows.
\item Since $Z(E)$ is abelian, the first statement is clear. Now for all $g\in E$
\begin{align*}
&\sum_{h\in E/C_E(g)}h^{-1}ghe_\varphi=\sum_{h\in E/C_E(g)}gg^{-1}h^{-1}ghe_\varphi\\
=&\sum_{h\in E/C_E(g)}g\varphi([g,h])e_\varphi=g\sum_{h\in [g,E]}\varphi(h)e_\varphi.
\end{align*}
Since $\varphi$ is faithful, this is zero unless $g\in Z(E)$. In other words, $Z(KEe_\varphi)=KZ(E)e_\varphi$. So the $e_\psi$'s, as $\psi$ ranges over $\Irr(Z(E),\varphi)$, are all the character idempotents of $KEe_\varphi$. Setting $\chi_\psi\in\Irr(E,\varphi)$ to be such that $e_{\chi_\psi}=e_\psi$, the claim follows by considering the left $KEe_\varphi$-module isomorphism
\begin{align*}
KE\otimes_{Z(E)}KZ(E)e_\psi\cong KEe_\psi=KEe_{\chi_\psi}
\end{align*}
and the left $KZ(E)e_\varphi$-module isomorphism
\begin{align*}
KZ(E)\otimes_{Z(E)}KEe_{\chi_\psi}=KZ(E)\otimes_{Z(E)}KEe_\psi\cong KZ(E)e_\psi\otimes_{Z(E)}KE,
\end{align*}
for all $\psi\in\Irr(Z(E),\varphi)$.
\newline
\newline
Let $\psi\in\Irr(Z(E),\varphi)$ and $\theta\in\Irr(E,1_Z)$, then
\begin{align*}
\psi\uparrow^E.\theta=(\psi.(\theta\downarrow_{Z(E)}))\uparrow^E
\end{align*}
and $\psi.(\theta\downarrow_{Z(E)})=\psi$ if and only if $\theta\in\Irr(E,1_{Z(E)})$. The final claim now follows from the previous paragraph.
\end{enumerate}
\end{proof}

We now describe $\Irr(B)$. Let $\lambda\in\Irr(D)$ and set $E_\lambda\leq E$ to be the stabiliser of $\lambda$ in $E$. Choose $\chi\in\Irr(E_\lambda)$ and define $(\lambda,\chi)\in\Irr(D\rtimes E_\lambda)$ by
\begin{align*}
(\lambda,\chi)(gh)=\lambda(g)\chi(h),
\end{align*}
for $g\in D$ and $h\in E_\lambda$. Note that $\ker(\lambda)\rtimes E_\lambda$ is a normal subgroup of $D\rtimes E_\lambda$ and so we can uniquely extend $\lambda$ to a character of $D\rtimes E_\lambda$ with kernel $\ker(\lambda)\rtimes E_\lambda$. $(\lambda,\chi)$ is just this extension tensored with the inflation of $\chi$ to $D\rtimes E_\lambda$.

\begin{lem}\label{lem:charG}$ $
\begin{enumerate}
\item The irreducible characters of $B$ are precisely of the form $(\lambda,\chi)\uparrow^G$ for some $\lambda\in\Irr(D)$ and $\chi\in\Irr(E_\lambda,\varphi)$.
\item $(\lambda_1,\chi_1)\uparrow^G=(\lambda_2,\chi_2)\uparrow^G$ if and only if exists $h\in E$ such that $\lambda_1^h=\lambda_2$ and $\chi_1^h=\chi_2$.
\end{enumerate}
\end{lem}

\begin{proof}$ $
\begin{enumerate}
\item Let $\lambda\in\Irr(D)$. Certainly $(\lambda,\chi)\in\Irr(D\rtimes E_\lambda,\lambda)$, for every $\chi\in\Irr(E_\lambda)$. Moreover, by considering restrictions to $E_\lambda$, distinct $\chi$'s give distinct $(\lambda,\chi)$'s. Now,
\begin{align*}
\dim_K(K(D\rtimes E_\lambda) e_\lambda)=&|E_\lambda|=\sum_{\chi\in E_\lambda}\chi(1)^2=\sum_{\chi\in E_\lambda}(\lambda,\chi)(1)^2\\
=&\dim_K\left(\bigoplus_{\chi\in E_\lambda}KE_\lambda e_{(\lambda,\chi)}\right).
\end{align*}
Therefore,
\begin{align*}
\Irr(E_\lambda,\lambda)=\{(\lambda,\chi)|\chi\in \Irr(E_\lambda\}.
\end{align*}
It now follows from~\cite[Theorem 6.11(b)]{is76} that \begin{align*}
\Irr(G)=\{(\lambda,\chi)\uparrow^G|\lambda\in\Irr(D),\chi\in \Irr(E_\lambda\}.
\end{align*}
The claim follows by noting that $(\lambda,\chi)\uparrow^G\in\Irr(B)$ if and only if
\begin{align*}
[E:E_\lambda]\chi(e_\varphi)=(\lambda,\chi)\uparrow^G(e_\varphi)\neq 0.
\end{align*}
\item If $\lambda_1^h=\lambda_2$ and $\chi_1^h=\chi_2$, for some $h\in E$, then \begin{align*}
(\lambda_2,\chi_2)\uparrow^G=(\lambda_1^h,\chi_1^h)\uparrow^G=(\lambda_1,\chi_1)^h\uparrow^G=(\lambda_1,\chi_1)\uparrow^G.
\end{align*}
Conversely if $(\lambda_1,\chi_1)\uparrow^G=(\lambda_2,\chi_2)\uparrow^G$, then, by restricting both sides to $D$ and considering irreducible constituents, $\lambda_1$ and $\lambda_2$ must be conjugate by an element of $E$ and so we may assume that $\lambda_1=\lambda_2$. Now, $(\lambda_1,\chi_1)$ is the unique irreducible character of $D\rtimes E_{\lambda_1}$ lying between $\lambda_1$ and $(\lambda_1,\chi_1)\uparrow^G$ and the same statement holds for $(\lambda_2,\chi_2)$. Therefore, $(\lambda_1,\chi_1)=(\lambda_2,\chi_2)$ and restricting both sides to $E_{\lambda_1}=E_{\lambda_2}$ yields that $\chi_1=\chi_2$.
\end{enumerate}
\end{proof}

For a finite group $H$, we write $H_{p'}$ for the set of $p$-regular elements of $H$.

\begin{lem}\label{lem:brauer}$ $
\begin{enumerate}
\item There is a bijection
\begin{align*}
\Irr(E,\varphi)&\to\IBr(B)\\
\chi&\mapsto\psi_\chi,
\end{align*}
where $\psi_\chi(g)=\Inf_E^G(\chi)(g)$, for all $g\in G_{p'}$ and $\Inf_E^G$ denotes the inflation of a character from $E\cong G/(D\times Z)$ to $G$.
\item Through this bijection we can identify the decomposition map \begin{align*}
\mathbb{Z}\Irr(B)\to\mathbb{Z}\IBr(B)
\end{align*}
with the restriction map
\begin{align*}
\mathbb{Z}\Irr(B)\to\mathbb{Z}\Irr(E,\varphi).
\end{align*}
\end{enumerate}
\end{lem}

\begin{proof}$ $
\begin{enumerate}
\item Every simple $kG$-module must have $D$ in its kernel and the decomposition map is a bijection on $p'$-groups, so we can associate a unique irreducible character of $E$ to each irreducible Brauer character of $G$. The first part then follows from the fact that an irreducible Brauer character $\psi$ of $G$ is in $B$ if and only if its restriction to $Z$ is $\varphi^{\oplus \psi(1)}$.
\item Every $\chi\in\Irr(B)$ restricted to $Z$ is $\varphi^{\oplus \chi(1)}$ and so the restriction map is well-defined. The claim now follows by noting that irreducible Brauer characters of $B$ are completely determined by their restriction to $E$.
\end{enumerate}
\end{proof}

\begin{cor}\label{cor:sigma}
Given any Morita auto-equivalence of $B$ with corresponding permutation $\sigma$ of $\Irr(B)$, there exists a unique permutation $\sigma_{\Br}$ of $\Irr(E,\varphi)$ that, when we extend to a $\mathbb{Z}$-linear endomorphism of $\mathbb{Z}\Irr(E,\varphi)$, satisfies
\begin{align*}
\sigma(\chi)\downarrow_E=\sigma_{\Br}(\chi\downarrow_E),
\end{align*}
for all $\chi\in\Irr(B)$.
\end{cor}

\begin{proof}
The existence of such a $\sigma_{\Br}$ is just the statement that any Morita auto-equivalence permutes $\IBr(B)$, which we identify with $\Irr(E,\varphi)$ via Lemma~\ref{lem:brauer}. The uniqueness follows from the fact that every element of $\Irr(E,\varphi)$ inflates to an element of $\Irr(B)$.
\end{proof}

By Lemma~\ref{lem:decompP} we may decompose $D=[D,Z(E)]\times C_D(Z(E))$.

\begin{lem}\label{lem:ZEWeiss}
The subset of irreducible characters of $B$ that reduce to some number of copies of the same irreducible Brauer character is $\Irr(B,1_{[D,Z(E)]})$.
\end{lem}

\begin{proof}
Let $\lambda\in\Irr(D)$ and $\chi\in\Irr(E,\varphi)$. Since $[D,Z(E)]$ is normal in $G$, $(\lambda,\chi)\uparrow^G\in\Irr(B,1_{[D,Z(E)]})$ if and only if $\lambda\in\Irr(D,1_{[D,Z(E)]})$. As $D=[D,Z(E)]\times C_D(Z(E))$, Lemma~\ref{lem:actirr} implies that $\lambda\in\Irr(D,1_{[D,Z(E)]})$ if and only if $Z(E)\leq E_\lambda$.
\newline
\newline
Part (2) of Lemma~\ref{lem:brauer} gives that $(\lambda,\chi)\uparrow^G$ reduces to some number of copies of the same irreducible Brauer character if and only if $(\lambda,\chi)\uparrow^G\downarrow_E=\chi\uparrow^E$ is the sum of some number of the same irreducible character of $E$. By part (2) of Lemma~\ref{lem:tenchar}, this happens if and only if $\chi\uparrow^E\downarrow_{Z(E)}$ is the sum of some number of the same irreducible character of $Z(E)$. It follows from the Mackey decomposition that $\chi\uparrow^E\downarrow_{Z(E)}$ is just multiple copies of $(\chi\downarrow_{Z(E)\cap E_\lambda})\uparrow^{Z(E)}$. Finally, since $Z(E)$ is abelian $(\chi\downarrow_{Z(E)\cap E_\lambda})\uparrow^{Z(E)}$ is the sum of some number of the same irreducible character of $Z(E)$ if and only if $Z(E)\leq E_\lambda$.
\end{proof}

Before proceeding we note that if $\omega\in\mathbb{C}$ is a primitive $(p^n)^{\nth}$-root of unity, for some $n\in\mathbb{N}$, then
\begin{align*}
\prod_{i=0,p\nmid i}^{p^n-1}(X-\omega^i)=(X^{p^n}-1)/(X^{p^{n-1}}-1)=\sum_{i=0}^{p-1}X^{ip^{n-1}}\in\mathbb{Q}[X].
\end{align*}
In particular,
\begin{align*}
\prod_{i=0,p\nmid i}^{p^n-1}(1-\omega^i)=p
\end{align*}
and so $1-\omega\in p\cO$ if and only if $p=2$ and $\omega=-1$.
\newline
\newline
The final lemma of this section is rather technical and will not be used until $\S$\ref{sec:one}. We set $\cO_p:=\cO/p\cO$ and $\cO_I:=\cO/I$, where $I:=\cJ.p\cO$ for $\cJ$ the unique maximal ideal of $\cO$. Recall from $\S$\ref{sec:abab} that we denote by $\Phi(P)$ the Frattini subgroup of $P$, for a finite abelian $p$-group $P$.

\begin{lem}\label{lem:Op}
Let $\lambda\in\Irr(D)$ and $\chi\in \Irr(E_\lambda,\varphi)$.
\begin{enumerate}
\item If $p$ is odd, then there exists an $\cO$-free $\cO G$-module $V$ affording $(\lambda,\chi)\uparrow^G$ with $x$ acting as the identity on $\cO_p\otimes_{\cO}V$, for all $x\in D_1$, if and only if $\lambda\downarrow_{D_1}=1_{D_1}$.
\item Let $p=2$.
\begin{enumerate}
\item There exists an $\cO$-free $\cO G$-module $V$ affording $(\lambda,\chi)$ with $x$ acting as the identity on $\cO_2\otimes_{\cO}V$, for all $x\in D_1$, if and only if $\lambda\downarrow_{\Phi(D_1)}=1_{\Phi(D_1)}$.
\item There exists an $\cO$-free $\cO G$-module $V$ affording $(\lambda,\chi)\uparrow^G$ with $x$ acting as the identity on $\cO_I\otimes_{\cO}V$, for all $x\in D_1$, if and only if $\lambda\downarrow_{D_1}=1_{D_1}$.
\end{enumerate}
\end{enumerate}
\end{lem}

\begin{proof}$ $
\begin{enumerate}
\item If such a $V$ exists then certainly $\lambda(x)\equiv 1\mod p$ for all $x\in D_1$. However, $1-\omega\in p\cO$ for some $p^{\nth}$-power root of unity $\omega\in\cO$ if and only if $\omega=1$. Therefore $\lambda=1_{D_1}$. Conversely suppose $\lambda=1_{D_1}$ and let $U$ be an $\cO$-free $\cO(D\rtimes E_\lambda)$-module affording $(\lambda,\chi)$. Certainly $x$ acts as the identity on $\cO_p\otimes_{\cO}U$, for all $x\in D_1$ and therefore setting $V:=U\uparrow^G$ proves the claim.
\item
\begin{enumerate}
\item The argument is identical for the $p=2$ case except that $1-\omega\in 2\cO$ for some $2^{\nd}$-power root of unity $\omega\in\cO$ if and only if $\omega=\pm1$. Therefore, such a $V$ exists if and only if $\lambda(x)=\pm1$ for all $x\in D_1$. In other words, if and only if $\lambda\downarrow_{\Phi(D_1)}=1_{\Phi(D_1)}$.
\item Again the result follows from the fact that $1-\omega\in I$ for some $2^{\nd}$-power root of unity $\omega\in\cO$ if and only if $\omega=1$ (now $1-(-1)\notin I$).
\end{enumerate}
\end{enumerate}
\end{proof}

\section{Perfect isometries}\label{sec:perfisom}

Let $H$ be a finite group and $b$ a block of $\cO H$. We write $\prj(b)$ for the set of characters of projective indecomposable $b$-modules.

\begin{defi}[\cite{br90}]
We denote by $\CF(H,b,K)$ the $K$-subspace of class functions on $H$ spanned by $\Irr(b)$, by $\CF(H,b,\cO)$ the $\cO$-submodule
\begin{align*}
\{\chi\in \CF(H,b,K):\chi(h)\in\cO\text{ for all }h\in H\}
\end{align*}
of $\CF(H,b,K)$ and by $\CF_{p'}(H,b,\cO)$ the $\mathcal{O}$-submodule
\begin{align*}
\{\phi\in \CF(H,b,\cO):\chi(h)=0\text{ for all }h\in H\backslash H_{p'}\}
\end{align*}
of $\CF(H,b,\mathcal{O})$.
\newline
\newline
Let $H'$ be another finite group and $b'$ a block of $\cO H'$. A perfect isometry between $b$ and $b'$ is an isometry
\begin{align*}
I:\mathbb{Z}\Irr(b)\to\mathbb{Z}\Irr(b'),
\end{align*}
such that
\begin{align*}
I_K:=K\otimes_{\mathbb{Z}}I:K\Irr(b)\to K\Irr(b'),
\end{align*}
induces an $\cO$-module isomorphism between $\CF(H,b,\cO)$ and $\CF(H',b',\cO)$ and also between $\CF_{p'}(H,b,\cO)$ and $\CF_{p'}(H',b',\cO)$. (Note that by an isometry we mean an isometry with respect to the usual inner products on $\mathbb{Z}\Irr(b)$ and $\mathbb{Z}\Irr(b')$. In particular, for all $\chi\in\Irr(b)$, $I(\chi)=\pm\chi'$ for some $\chi'\in\Irr(b')$).
\end{defi}

\begin{rem}
An alternative way of phrasing the condition that $I_K$ induces an isomorphism between $\CF_{p'}(H,b,\cO)$ and $\CF_{p'}(H',b',\cO)$ is that $I$ induces an isomorphism $\mathbb{Z}\prj (b) \cong \mathbb{Z}\prj (b')$.
\end{rem}

We need a small lemma before continuing.

\begin{lem}\label{lem:NO}
Let $H$ be a finite group, $N$ a normal subgroup and $\chi\in\mathbb{Z}\Irr(N)$. Then $\chi\uparrow^H(g)\in[\Stab_H(\chi):N]\cO$, for all $g\in H$.
\end{lem}

\begin{proof}
$\chi^H$ is zero outside of $N$ so we need only prove that $\chi\uparrow^H(g)\in[\Stab_H(\chi):N]\cO$, for all $g\in N$. This is now clear, as $\chi\uparrow^H\downarrow_N$ is just the sum of the $H$-conjugates of $\chi$, each appearing with multiplicity $[\Stab_H(\chi):N]$.
\end{proof}

The following lemma is the main result of this section.

\begin{lem}\label{lem:D2kernel}
Let $\sigma$ be a permutation of $\Irr(B)$ induced by a Morita auto-equivalence of $B$. Identifying $(D_1\rtimes E)\times D_2$ with $D\rtimes E$, there exists $\theta\in\Irr(D_2)$ such that
\begin{align}\label{algn:ker}
\sigma(1_D,\chi)=\psi_\chi\otimes\theta,
\end{align}
for all $\chi\in\Irr(E,\varphi)$, where $\psi_\chi\in\Irr(D_1\rtimes E,\varphi)$.
\end{lem}

\begin{proof}
Let's fix some $\xi\in\Irr(E,\varphi)$ and define $\theta\in\Irr(D_2)$ by
\begin{align*}
\sigma(1_D,\xi)=\psi_\xi\otimes\theta,
\end{align*}
where $\psi_\xi\in\Irr(D_1\rtimes E,\varphi)$.
\newline
\newline
Let $D'\lhd D_1\rtimes E$ be properly contained in $D_1$ and maximal with respect to these two conditions. Define $E'\leq E$ to be the subgroup inducing the identity on $D_1/D'$. In particular, $D_1/D'$ is elementary abelian and $E/E'$ acts indecomposably on $D_1/D'$. Therefore, by part (1) of Lemma~\ref{lem:jacobi}, $E/E'$ is cyclic and, by part (2) of the same Lemma and Lemma~\ref{lem:actirr}, $E'$ is the stabiliser in $E$ of any non-trivial character of $D_1/D'$ inflated to $D_1$.
\newline
\newline
We claim that the the set of $\chi\in\Irr(E,\varphi)$ that satisfy (\ref{algn:ker}) is closed under tensoring with elements of $\Irr(E,1_{E'})$. Let $\tau\in\Irr(E',\varphi)$. Note that if $\tau\uparrow^E$ is irreducible then $\tau\uparrow^E$ (and by part (1) of Lemma~\ref{lem:tenchar} every character of $\Irr(E,\varphi)$) is fixed under tensoring with elements of $\Irr(E,1_{E'})$ and there is nothing to prove. So let's assume $\tau\uparrow^E$ is reducible. Next consider
\begin{align*}
(1_{D'\times D_2},\tau)\in\Irr((D'\times D_2)\rtimes E').
\end{align*}
Since $(D'\times D_2)\rtimes E'\lhd G$ and $\Stab_G(1_{D'\times D_2},\tau)\geq D\rtimes E'$, it follows from Lemma~\ref{lem:NO} that
\begin{align}\label{algn:D1:D}
(1_{D'\times D_2},\tau)\uparrow^G\in[D_1:D']\CF(G,B,\mathcal{O}).
\end{align}
Since $E/E'$ is cyclic, a direct calculation gives
\begin{align}\label{algn:sum}
\begin{split}
(1_{D'\times D_2},\tau)\uparrow^G&=(1_{D'\times D_2},\tau)\uparrow^{D\rtimes E'}\uparrow^G\\
&=\sum_{\chi\in\Irr(E,\tau)}(1_D,\chi)+\sum_{1_D\neq\lambda\in\Irr(D,1_{D'\times D_2})}(\lambda,\tau)\uparrow^G.
\end{split}
\end{align}
We define $\sigma_{\Br}$ as in Corollary~\ref{cor:sigma} and set $X:=\sigma_{\Br}(\Irr(E,\tau))$. Since $\sigma(1_D,\chi)\downarrow_E=\chi$,
\begin{align}\label{algn:br}
\sigma(1_D,\chi)\downarrow_{D_2\times E}=\theta_\chi\otimes\sigma_{\Br}(\chi),
\end{align}
for all $\chi\in\Irr(E,\tau)$, where $\theta_\chi\in\Irr(D_2)$. Again, since $E/E'$ is cyclic,
\begin{align*}
\sigma((\lambda,\tau)\uparrow^G)\downarrow_E=\sigma_{\Br}((\lambda,\tau)\uparrow^G\downarrow_E)=\sigma_{\Br}(\tau\uparrow^E)=\sum_{\eta\in X}\eta,
\end{align*}
for all $1\neq\lambda\in\Irr(D,1_{D'\times D_2})$. So
\begin{align*}
\sigma((\lambda,\tau)\uparrow^G)=(\lambda_\sigma,\zeta_{\lambda_\sigma})\uparrow^G=(\lambda_{\sigma,1},\zeta_{\lambda_\sigma})\uparrow^{D_1\rtimes E}\otimes\lambda_{\sigma,2}&&\text{and}&&\zeta_{\lambda_\sigma}\uparrow^E=\sum_{\eta\in X}\eta,
\end{align*}
for some $\lambda_\sigma\in\Irr(D)$, where $\lambda_\sigma=\lambda_{\sigma,1}\otimes\lambda_{\sigma,2}$, for $\lambda_{\sigma,1}\in\Irr(D_1)$, $\lambda_{\sigma,2}\in\Irr(D_2)$ and $\zeta_{\lambda_\sigma}\in\Irr(E_{\lambda_\sigma},\varphi)$. In particular,
\begin{align}\label{algn:br2}
\sigma((\lambda,\tau)\uparrow^G)(x(e_{\sigma_{\Br}(\chi_1)}-e_{\sigma_{\Br}(\chi_2)}))=0,
\end{align}
for all $x\in D_2$ and $\chi_1\neq\chi_2\in \Irr(E,\tau)$. Now~\cite[Th\'eor\`eme 1.2]{br90} implies that $\sigma$ induces a perfect self-isometry of $B$. So plugging $x(e_{\sigma_{\Br}(\chi_1)}-e_{\sigma_{\Br}(\chi_2)})$ into $\sigma$ applied to (\ref{algn:sum}) and applying (\ref{algn:D1:D}), (\ref{algn:br}) and (\ref{algn:br2}) gives
\begin{align*}
\theta_{\chi_1}(x)\sigma_{\Br}(\chi_1)(e_{\sigma_{\Br}(\chi_1)})-\theta_{\chi_2}(x)\sigma_{\Br}(\chi_2)(e_{\sigma_{\Br}(\chi_2)})\in[D_1:D']\cO,
\end{align*}
for all $\chi_1\neq\chi_2\in \Irr(E,\tau)$ and $x\in D_2$. Now part (1) of Lemma~\ref{lem:tenchar} and the fact that $E$ is a $p'$-group imply that $\sigma_{\Br}(\chi_1)(e_{\sigma_{\Br}(\chi_1)})=\sigma_{\Br}(\chi_2)(e_{\sigma_{\Br}(\chi_2)})\in\cO^\times$. Therefore,
\begin{align*}
\theta_{\chi_1}(x)-\theta_{\chi_2}(x)\in[D_1:D']\cO.
\end{align*}
It follows from the comments preceding Lemma~\ref{lem:Op} that $\theta_{\chi_1}(x)=\theta_{\chi_2}(x)$ unless $[D_1:D']=2$. However, if $[D_1:D']=2$ then $E'=E$ and there is nothing to prove. Since $x\in D_2$ was arbitrary, $\theta_{\chi_1}=\theta_{\chi_2}$. Now, since (\ref{algn:br}) gives that
\begin{align*}
\sigma(1_D,\chi_1)=\psi_{\chi_1}\otimes\theta_{\chi_1}\text{ and }\sigma(1_D,\chi_2)=\psi_{\chi_2}\otimes\theta_{\chi_2},
\end{align*}
for some $\psi_{\chi_1},\psi_{\chi_2}\in\Irr(D_1\rtimes E,\varphi)$, we have proved that if (\ref{algn:ker}) holds for one $\chi\in\Irr(E,\tau)$ it holds for all of them. Since the choice of $\tau\in\Irr(E',\varphi)$ was arbitrary we have proved that the set of $\chi\in\Irr(E,\varphi)$ satisfying (\ref{algn:ker}) is closed under tensoring with elements of $\Irr(E,1_{E'})$.
\newline
\newline
In the final part of the proof we prove that the intersection of all possible choices for $E'$ is $Z$. We will have then proved that the set of $\chi\in\Irr(E,\varphi)$ that satisfy (\ref{algn:ker}) is closed under tensoring with elements of $\Irr(E,1_Z)$. By part (1) of Lemma~\ref{lem:tenchar} we will then be done.
\newline
\newline
Let's decompose
\begin{align*}
D_1=Q_1\times\dots\times Q_t,
\end{align*}
where $E/C_E(Q_i)$ acts indecomposably on $Q_i$. Now Lemma~\ref{lem:fratt} implies that $E/C_E(Q_i)$ also acts indecomposably on each $Q_i/\Phi(Q_i)$. In particular, for each $1\leq i\leq t$,
\begin{align*}
\left(\prod_{j\neq i}Q_j\right)\times\Phi(Q_i)
\end{align*}
is a valid choice for $D'$ and $C_E(Q_i)$ a valid choice for $E'$. Finally, by the definition of $Z$,
\begin{align*}
\bigcap_{i}C_E(Q_i)=C_E(D_1)=C_E(D)=Z
\end{align*}
and the proof is complete.
\end{proof}

\section{Blocks with one simple module}\label{sec:one}

Throughout this section we assume that $B$ has, up to isomorphism, a unique simple module. By part (1) of Lemma~\ref{lem:brauer} and part (2) of Lemma~\ref{lem:tenchar}, this implies that $Z=Z(E)$. For each $g\in L=E/Z$ we define
\begin{align*}
\phi_g:L&\to Z\\
h&\mapsto [\tilde{g},\tilde{h}],
\end{align*}
where $\tilde{g}$ and $\tilde{h}$ represent lifts to $E$ of $g$ and $h$ respectively. Note it is easy to check that $\phi_g$ is a well-defined group homomorphism.

\begin{lem}\label{lem:duality}
\begin{align*}
L&\to\Hom(L,\cO^\times)\\
g&\mapsto\varphi\circ\phi_g
\end{align*}
is an isomorphism of groups.
\end{lem}

\begin{proof}
This is just~\cite[Lemma 4.1]{hk05} and its proof.
\end{proof}

We now introduce some further notation. First decompose
\begin{align*}
D=P_1\times\dots\times P_n,
\end{align*}
where $E/C_E(P_i)$ acts on each $P_i$ indecomposably and $P_i\cong (C_{p^{n_i}})^{m_i}$. We choose this decomposition such that
\begin{align*}
D_1=P_1\times\dots\times P_t&&\text{and}&&D_2=P_{t+1}\times\dots\times P_n,
\end{align*}
for some $1\leq t\leq n$. In particular, $m_i=1$ for all $i>t$. We now state and prove a partial analogue of~\cite[Theorem 1.1(i)]{bg04} and~\cite[Corollary 4.3]{hk05} over $\cO$. Note we do not describe the basic algebra of $B$ exactly, in contrast to~\cite[Theorem 1.1(i)]{bg04} and~\cite[Corollary 4.3]{hk05}, where the basic algebra of $k\otimes_{\cO}B$ (for $D$ elementary abelian and arbitrary abelian respectively) is completely described.

\begin{lem}\label{lem:basic}
There exists an $\cO$-algebra $A$ with the following properties:
\begin{enumerate}
\item $B\cong M_d(\cO)\otimes_{\cO}A$, where $d$ is the dimension of the unique simple $B$-module. In particular, $A$ is basic.
\item There exist $X_{ij}\in A$, for $1\leq i\leq n$ and $1\leq j\leq m_i$ that generate $A$ as an $\cO$-algebra. Furthermore,
\begin{align*}
\cB:=\left\{\prod_{i=1}^n\prod_{j=1}^{m_i} X_{ij}^{l_{ij}}|0\leq l_{ij}<p^{n_i},\text{ for all }1\leq i\leq n,1\leq j\leq m_i\right\}
\end{align*}
forms an $\cO$-basis for $A$.
\item There exist $p'$-roots of unity $q_{i_1j_1,i_2j_2}\in \cO^\times$ such that
\begin{align*}
X_{ij}^{p^{n_i}}\in pA,&&X_{i_1j_1}X_{i_2j_2}=q_{i_1j_1,i_2j_2}X_{i_2j_2}X_{i_1j_1},
\end{align*}
for all $1\leq i,i_1,i_2\leq n$ and $1\leq j\leq m_i$, $1\leq j_1\leq m_{i_1}$, $1\leq j_2\leq m_{i_2}$. Moreover,
\begin{enumerate}
\item $q_{ij_1,ij_2}=1$, for all $1\leq i\leq n$ and $1\leq j_1,j_2\leq m_i$.
\item $q_{i_1j_1,i_2j_2}q_{i_2j_2,i_1j_1}=1$, for all $1\leq i_1,i_2\leq n$ and $1\leq j_1\leq m_{i_1}$, $1\leq j_2\leq m_{i_2}$.
\item $q_{i_1j_1,i_2(j_2+1)}=q_{i_1j_1,i_2j_2}^p$, for all $1\leq i_1,i_2\leq n$ and $1\leq j_1\leq m_{i_1}$, $1\leq j_2\leq m_{i_2}$, where $j_2+1$ is considered modulo $m_{i_2}$.
\item Let $1\leq i_1\leq n$ and $1\leq j_1\leq m_{i_1}$, then $q_{i_1j_1,ij}=1$ for all $1\leq i\leq n$ and $1\leq j\leq m_i$ if and only if $i_1>t$.
\item Let $1\leq i_1\leq t$ and $1\leq j_1\leq m_{i_1}$. For all $1\leq j_2\leq m_{i_1}$ with $j_2\neq j_1$, there exist $1\leq i\leq t$ and $1\leq j\leq m_i$ such that $q_{i_1j_1,ij}\neq q_{i_1j_2,ij}$.
\end{enumerate}
\item We can choose the $\cO$-algebra isomorphism $B\to M_d(\cO)\otimes_{\cO}A$ such that we have the following identification of ideals
\begin{align*}
((1-x)e_\varphi)_{x\in D_1}=M_d(\cO)\otimes_{\cO}T,
\end{align*}
where $T:=(X_{ij})_{\substack{1\leq i\le t\\ 1\leq j\leq m_i}}\lhd A$.
\item For all $i>t$,
\begin{align*}
\left\langle X_{i1}^{l_i}|0\leq l_i<p^{n_i}\right\rangle_{\cO}
\end{align*}
is an $\cO$-subalgebra of $A$. Moreover, if $p=2$ and $D_1$ is elementary abelian, then for each $1\leq i\leq t$ and $1\leq j\leq m_i$, $m_i>1$ and
\begin{align*}
X_{ij}^2-2X_{i(j+1)}\in 2\left\langle\prod_{l=1}^{m_i} X_{il}^{\epsilon_{il}}|\epsilon_{il}\in\{0,1\},\sum_{l=1}^{m_i}\epsilon_{il}>1\right\rangle_{\cO},
\end{align*}
where $j+1$ is considered modulo $m_i$. 
\end{enumerate}
\end{lem}

\begin{proof}$ $
\begin{enumerate}
\item By part (1) of Lemma~\ref{lem:brauer}, $|\Irr(E,\varphi)|=1$ and so $\cO Ee_\varphi\cong M_d(\cO)$, where $d$ is the dimension of the unique simple $B$-module. Therefore, $\cO Ee_\varphi$ is a central simple subalgebra of $B$ and so we have $B\cong \cO Ee_\varphi\otimes_{\cO}C_B(\cO Ee_\varphi)$. We, therefore, define $A:=C_B(\cO Ee_\varphi)$.
\item By parts (1) and (3) of Lemma~\ref{lem:jacobi}, for each $1\leq i\leq n$, $J(kP_i)/J^2(kP_i)$ decomposes into $m_i$ non-isomorphic linear representations of $E$,
\begin{align*}
\{\rho_i=\rho_i^{p^m_i},\rho_i^p,\dots,\rho_i^{p^{m_i-1}}\},
\end{align*}
with respect to the conjugation action of $E$ on $P_i$. As $p\nmid|E|$, we can decompose
\begin{align*}
J(kP_i)=\langle w_{i1}\rangle_k\oplus\dots\oplus\langle w_{im_i}\rangle_k\oplus J^2(kP_i),
\end{align*}
into $kE$-modules, where $\langle w_{ij}\rangle_k$ affords the representation $\rho_i^{p^{j-1}}$. Again, since $p\nmid|E|$ and $k\otimes_{\cO}J_{\cO}(P_i)\cong J(kP_i)$, we can lift $w_{ij}$ to $W_{ij}\in J_{\cO}(P_i)$ such that $\langle W_{ij}\rangle_{\cO}$ affords the representation $\varrho_i^{p^{j-1}}$, where $\varrho_i$ is the unique lift of $\rho_i$ to a representation of $E$ over $\cO$. Certainly $Z$ commutes with $D$ and so, by Lemma~\ref{lem:duality}, we can choose $h_{i1}\in E$ such that $\varrho_i=\varphi\circ\phi_{h_{i1}Z}$. Setting $h_{ij}:=h_{i1}^{p^{j-1}}$ gives that $\varrho_i^{p^{j-1}}=\varphi\circ\phi_{h_{ij}Z}$, for $1\leq j\leq m_i$. Note that, since $\varrho_i^{p^{n_1}}=\varrho_i$, $h_{i1}^{p^{n_i}}Z=h_{i1}Z$.
\newline
\newline
Now set $X_{ij}=h_{ij}W_{ij}e_\varphi$, for all $1\leq i\leq n$ and $1\leq j\leq m_i$. We first note that
\begin{align*}
hX_{ij}&=hh_{ij}W_{ij}e_\varphi=hh_{ij}h^{-1}hW_{ij}h^{-1}he_\varphi=h_{ij}[h_{ij},h^{-1}](\varrho_i^{p^{j-1}}(h)W_{ij})he_\varphi\\
&=h_{ij}\varphi([h_{ij},h^{-1}])(\varrho_i^{p^{j-1}}(h)W_{ij})he_\varphi=h_{ij}\varrho_i^{p^{j-1}}(h^{-1})(\varrho_i^{p^{j-1}}(h)W_{ij})he_\varphi\\
&=h_{ij}W_{ij}he_\varphi=X_{ij}h,
\end{align*}
for all $h\in E$ and so $X_{ij}\in C_B(\cO Ee_\varphi)$. Note that the $\overline{X}_{ij}=h_{ij}w_{ij}e_\varphi\in k\otimes_{\cO}A$ are precisely the $X_i$'s constructed in the proof of~\cite[Corollary 4.3]{hk05}. In particular, $k\otimes_{\cO}\cB$ forms a basis for $C_{kB}(kEe_\varphi)$. So $\cB$ is an $\cO$-linearly independent set and $\langle \cB\rangle_{\cO}$ is an $\cO$-summand of $B$. Therefore, since $\langle\cB\rangle_{\cO}\subseteq C_B(\cO Ee_\varphi)$ and
\begin{align*}
d^2.\rk_{\cO}(\langle\cB\rangle_{\cO})=&d^2.\dim_k(\langle k\otimes_{\cO}\cB\rangle_k)=\dim_k(B)\\
=&\rk_{\cO}(B)=d^2.\rk_{\cO}(C_B(\cO Ee_\varphi)),
\end{align*}
we have that $\langle\cB\rangle_{\cO}=C_B(\cO Ee_\varphi)$.
\item By Lemma~\ref{lem:tech}, $W_{ij}^{p^{n_i}}\in pJ_{\cO}(P_i)$ for all $i$ and $j$. Therefore, since
\begin{align*}
h_{ij}W_{ij}h_{ij}^{-1}=\varrho_i^{p^{j-1}}(h_{ij})W_{ij}=\varphi([h_{ij},h_{ij}])W_{ij}=W_{ij}
\end{align*}
and $A$ is an $\cO$-summand of $B$, we have
\begin{align*}
X_{ij}^{p^{n_i}}\in pB\cap A=pA.
\end{align*}
Next
\begin{align*}
X_{i_1j_1}X_{i_2j_2}e_{\varphi}&=h_{i_1j_1}W_{i_1j_1}h_{i_2j_2}W_{i_2j_2}e_{\varphi}=\varrho_{i_1}^{p^{j_1-1}}(h_{i_2j_2}^{-1})h_{i_1j_1}h_{i_2j_2}W_{i_1j_1}W_{i_2j_2}e_{\varphi}\\
&=\varphi([h_{i_1j_1},h_{i_2j_2}^{-1}])\varphi([h_{i_1j_1}^{-1},h_{i_2j_2}^{-1}])h_{i_2j_2}h_{i_1j_1}W_{i_2j_2}W_{i_1j_1}e_\varphi\\
&=\varphi([h_{i_1j_1},h_{i_2j_2}^{-1}])\varphi([h_{i_1j_1}^{-1},h_{i_2j_2}^{-1}])\varrho_{i_2}^{p^{j_2-1}}(h_{i_1j_1})h_{i_2j_2}W_{i_2j_2}h_{i_1j_1}W_{i_1j_1}e_\varphi\\
&=\varphi([h_{i_1j_1},h_{i_2j_2}^{-1}])\varphi([h_{i_1j_1}^{-1},h_{i_2j_2}^{-1}])\varphi([h_{i_2j_2},h_{i_1j_1}])h_{i_2j_2}W_{i_2j_2}h_{i_1j_1}W_{i_1j_1}e_\varphi\\
&=\varphi([h_{i_2j_2},h_{i_1j_1}])X_{i_2j_2}X_{i_1j_1}e_\varphi,
\end{align*}
for all $1\leq i_1,i_2\leq n$ and $1\leq j_1\leq m_{i_1}$, $1\leq j_2\leq m_{i_2}$. We, therefore, set $q_{i_1j_1,i_2j_2}:=\varphi([h_{i_2j_2},h_{i_1j_1}])=\varrho_{i_2}^{p^{j_2-1}}(h_{i_1j_1})$. Parts (a) and (b) follow immediately from this definition. Part (c) holds since
\begin{align*}
\varphi([h_{i_2(j_2+1)},h_{i_1j_1}])=\varrho_{i_2}^{p^{j_2}}(h_{i_1j_1})=\varphi([h_{i_2j_2},h_{i_1j_1}])^p.
\end{align*}
As representations of $L$, we claim that the $\varrho_i$'s, for $1\leq i\leq t$, generate $\Hom(L,\cO^\times)$. Assume this is not the case and so they generate some proper subgroup of $\Hom(L,\cO^\times)$. Therefore, there exists some $\{1\}\neq L'\leq L$ such that $\varrho_i(l)=1$, for all $1\leq i\leq t$ and $l\in L'$. So, by the definition of the $\varrho_i$'s and Lemma~\ref{lem:jacobi}, $L'$ commutes with $P_i/\Phi(P_i)$ and therefore by Lemma~\ref{lem:fratt}, $L'$ commutes with $P_i$, for all $1\leq i\leq t$. This contradicts $Z:=C_E(D)=C_E(D_1)$.
\newline
\newline
For part (d), note that $q_{i_1j_1,ij}=1$ for all $1\leq i\leq n$ and $1\leq j\leq m_i$ if and only if $\varrho_i(h_{i_1j_1})=1$ for all such $i$ if and only if $h_{i_1j_1}\in Z$ (since the $\varrho_i$'s generate $\Hom(L,\cO^\times)$) if and only if $\varrho_{i_1}$ is trivial. However, by Lemmas~\ref{lem:fratt} and~\ref{lem:jacobi}, $\varrho_{i_1}$ is trivial if and only if $E$ acts trivially on $P_i$ i.e. $i_1>t$.
\newline
\newline
For part (e) we suppose the contrary, that is there exists $1\leq j_2\leq m_{i_1}$ with $j_1\neq j_2$ such that $q_{i_1j_1,ij}=q_{i_1j_2,ij}$ for all $1\leq i\leq t$ and $1\leq j\leq m_i$. In other words, $\varrho_i^{p^{j-1}}(h_{i_1j_1})=\varrho_i^{p^{j-1}}(h_{i_1j_2})$, for all such $i$ and $j$. Since the $\varrho_i$'s generate $\Hom(L,\cO^\times)$, this implies that $h_{i_1j_1}Z=h_{i_1j_2}Z$, which in turn implies $\varrho_{i_1}^{p^{j_1-1}}=\varrho_{i_1}^{p^{j_2-1}}$ contradicting part (1) of Lemma~\ref{lem:jacobi}.
\item First note that the $w_{ij}$'s, for $1\leq i\leq t$ and $1\leq j\leq m_i$, form a basis for
\begin{align*}
J(kD_1)/J^2(kD_1)\cong\bigoplus_{i=1}^t J(kP_i)/J^2(kP_i).
\end{align*}
Therefore,
\begin{align*}
(W_{ij})_{\substack{1\leq i\leq t \\ 1\leq j\leq m_i}}+J(\cO D_1)J_{\cO}(\cO D_1)=J_{\cO}(\cO D_1),
\end{align*}
as $\cO D_1$-modules, where $J_{\cO}(\cO D_1)=(1-x)_{x\in D_1}\lhd \cO D_1$. Nakayama's lemma now gives that the $W_{ij}$'s, for $1\leq i\leq t$ and $1\leq j\leq m_i$, generate $J_{\cO}(\cO D_1)$. Finally, noting that
\begin{align*}
W_{ij}e_\varphi=h_{ij}^{-1}e_\varphi\otimes X_{ij}\in \cO Ee_\varphi\otimes_{\cO}A,
\end{align*}
for all $1\leq i\leq t$ and $1\leq j\leq m_i$, we get that
\begin{align*}
((1-x)e_\varphi)_{x\in D_1}=\cO Ee_\varphi\otimes_{\cO} (X_{ij})_{\substack{1\leq i\leq t \\ 1\leq j\leq m_i}}\lhd \cO Ee_\varphi\otimes_{\cO}A,
\end{align*}
as required.
\item Similarly to the proof of part (4),
\begin{align}\label{algn:basis}
\left\{\prod_{j=1}^{m_i} W_{ij}^{l_{ij}}|0\leq l_{ij}<p^{n_i},\text{ for all }1\leq j\leq m_i\right\}
\end{align}
forms an $\cO$-span of $\cO P_i$, for all $1\leq i\leq n$. (Note that we need not take higher powers of $W_{ij}$ since, as noted in the proof of part (3), $W_{ij}^{p^{n_i}}\in pJ_{\cO}(\cO P_i)$.) By comparing $\cO$-ranks, (\ref{algn:basis}) must actually form a basis of $\cO P_i$. The first statement now follows from the fact that, since $E$ commutes with $P_i$, $h_{i1}\in Z$, for all $i>t$.
\newline
\newline
From now on we assume $p=2$ and $D_1$ is elementary abelian. We fix some $1\leq i\leq t$. If $m_i=1$, then $P_i$ has no non-trivial automorphisms and so $P_i\leq C_D(E)=D_2$, a contradiction. So we must have $m_i>1$. Let $1\leq j\leq m_i$, where we are considering $j$ modulo $m_i$. Of course, $n_i=1$ and so, by Lemma~\ref{lem:tech}, $W_{ij}^2=2y$ for some $y\in J_{\cO}(P_i)\backslash J_{\cO}(P_i)_2$. Therefore, $\langle y\rangle_{\cO}$ affords $\varrho_i^{2^j}$ and so, by the construction of the $W_{il}$'s and using the basis from (\ref{algn:basis}), $y=\lambda_j W_{i(j+1)}+W$, for some $\lambda_j\in\cO$ and
\begin{align*}
W\in\left\langle\prod_{l=1}^{m_i} W_{il}^{\epsilon_{il}}|\epsilon_{il}\in\{0,1\},\sum_{l=1}^{m_i}\epsilon_{il}>1\right\rangle_{\cO}.
\end{align*}
Since $y\notin J_{\cO}(P_i)_2$, in fact $\lambda_j\in\cO^\times$. By the construction of the $h_{il}$'s, we also have $h_{ij}^2Z=h_{i(j+1)}Z$ and so $h_{ij}^2e_\varphi=\mu_jh_{i(j+1)}e_\varphi$, for some $\mu_j\in\cO^\times$. Therefore,
\begin{align*}
X_{ij}^2=&(h_{ij}W_{ij}e_\varphi)^2=h_{ij}^2W_{ij}^2e_\varphi=2\mu_jh_{i(j+1)}(\lambda_jW_{i(j+1)}+W)e_\varphi\\
=&2\mu_j\lambda_jX_{i(j+1)}+2\mu_jh_{i(j+1)}We_\varphi.
\end{align*}
Let $\prod_{l=1}^{m_i} W_{il}^{\epsilon_{il}}$ appear with non-zero coefficient in $W$. Since $y$ affords the character $\varrho_i^{2^j}$, so do $\langle W\rangle_{\cO}$ and $\langle\prod_{l=1}^{m_i} W_{il}^{\epsilon_{il}}\rangle_{\cO}$. Therefore, by Lemma~\ref{lem:duality} and the construction of the $h_{il}$'s, $h_{i(j+1)}Z=\left(\prod_{l=1}^{m_i} h_{il}^{\epsilon_{il}}\right)Z$ and so
\begin{align*}
\left\langle h_{i(j+1)}\prod_{l=1}^{m_i} W_{il}^{\epsilon_{il}}e_\varphi\right\rangle_{\cO}=\left\langle\prod_{l=1}^{m_i} X_{il}^{\epsilon_{il}}\right\rangle_{\cO},
\end{align*}
for all $\prod_{l=1}^{m_i} W_{il}^{\epsilon_{il}}$ appearing with non-zero coefficient in $W$. We have now shown that
\begin{align*}
X_{ij}^2-2\lambda_j\mu_jX_{i(j+1)}=2\mu_jh_{i(j+1)}We_\varphi\in2\left\langle\prod_{l=1}^{m_i} X_{il}^{\epsilon_{il}}|\epsilon_{il}\in\{0,1\},\sum_{l=1}^{m_i}\epsilon_{il}>1\right\rangle_{\cO},
\end{align*}
for all $1\leq j\leq m_i$. Since raising to the power $2^{m_i}-1$ on $\cO^\times$ is a surjective map, there exists $\alpha_1\in\cO^\times$ such that
\begin{align*}
\alpha_1^{2^{m_i}-1}=(\lambda_1^{2^{m_i-1}}\mu_1^{2^{m_i-1}})^{-1}(\lambda_2^{2^{m_i-2}}\mu_2^{2^{m_i-2}})^{-1}\dots(\lambda_{m_i}\mu_{m_i})^{-1}.
\end{align*}
Now set $\alpha_{j+1}=\alpha_j^2\lambda_j\mu_j$, for all $1\leq j\leq m_i$. Note that by this definition
\begin{align*}
\alpha_{m_i+1}=\alpha_1^{2^{m_i}}(\lambda_1^{2^{m_i-1}}\mu_1^{2^{m_i-1}})(\lambda_2^{2^{m_i-2}}\mu_2^{2^{m_i-2}})\dots(\lambda_{m_i}\mu_{m_i})=\alpha_1.
\end{align*}
In other words, $\alpha_j$ is well-defined when we consider $j\mod m_i$. Therefore,
\begin{align*}
(\alpha_jX_{ij})^2-2\alpha_{j+1}X_{i(j+1)}\in2\left\langle\prod_{l=1}^{m_i} X_{il}^{\epsilon_{il}}|\epsilon_{il}\in\{0,1\},\sum_{l=1}^{m_i}\epsilon_{il}>1\right\rangle_{\cO},
\end{align*}
for all $1\leq j\leq m_i$. Let's replace $W_{ij}$ with $\alpha_jW_{ij}$ and therefore $X_{ij}$ with $\alpha_jX_{ij}$, for $1\leq j\leq m_i$. These new $X_{ij}$'s satisfy the required properties.
\end{enumerate}
\end{proof}

For what follows, recall that $T:=(X_{ij})_{\substack{1\leq i\le t\\ 1\leq j\leq m_i}}\lhd A$, $\cO_p:=\cO/p\cO$ and also $\cO_I:=\cO/I$, where $I:=\cJ.p\cO$, for $\cJ$ the unique maximal ideal of $\cO$. In addition we set $A_k:=k\otimes_\cO A$, $A_p:=\cO_p\otimes_\cO A$, $A_I:=\cO_I\otimes_\cO A$, $T_p:=\cO_p\otimes_{\cO}T\lhd A_p$ and $T_I:=\cO_I\otimes_{\cO}T\lhd A_I$. The following immediate consequence of Lemma~\ref{lem:basic} can be thought of an $\cO_p$ analogue to~\cite[Theorem 1.1(i)]{bg04} and~\cite[Corollary 4.3]{hk05}.

\begin{cor}\label{cor:Op}
\begin{align*}
A_p\cong\cO_p[X_{ij},1\leq i\leq n,1\leq j\leq m_i]/(X_{i_1j_1}X_{i_2j_2}=q_{i_1j_1,i_2j_2}X_{i_2j_2}X_{i_1j_1},X_{ij}^{p^{n_i}}=0).
\end{align*}
\end{cor}

We may express each element of $A_k$ uniquely as a $k$-linear combination of elements of $k\otimes_{\cO}\cB$ and in the following lemma we refer to the terms of this $k$-linear combination. Similarly we refer to the terms of an element of $A_p$ and $A_I$. Set
\begin{align*}
\mathbf{A}=\{(a_{t+1},\dots,a_n)|0\leq a_i<p^{n_i}\text{ for all }t+1\leq i\leq n\}.
\end{align*}
We denote by $X^\mathbf{a}$, the monomial $X_{t+1}^{a_{t+1}}\dots X_n^{a_n}\in k[X_{i1}]_{t+1\leq i\le n}$, where $\mathbf{a}=(a_{t+1},\dots,a_n)\in\mathbf{A}$. For $\mathbf{a},\mathbf{b}\in\mathbf{A}$, $\mathbf{a}+\mathbf{b}$ signifies the componentwise sum, when this is still in $\mathbf{A}$. We have a partial order on $\mathbf{A}$ given by $\mathbf{a}\leq \mathbf{c}$ if and only if there exists $\mathbf{b}\in\mathbf{A}$ such that $\mathbf{a}+\mathbf{b}=\mathbf{c}$. In this case note that $X^{\mathbf{a}+\mathbf{b}}=X^{\mathbf{a}}X^{\mathbf{b}}$. We adopt the same notation for monomials in $\cO_p[X_{i1}]_{t+1\leq i\le n}$ and $\cO_I[X_{i1}]_{t+1\leq i\le n}$.
\newline
\newline
In what follows, when it is clear from the context, we will denote by $X_{ij}$ its image in $A_k$, $A_p$ or $A_I$, for $1\leq i\leq n$ and $1\leq j\leq m_i$. The next lemma should be thought of as an attempt to generalise~\cite[Lemma 2.3]{bk07}, where automorphisms of $A_k$ are studied. In~\cite{bk07} the term special generating set is used to describe certain subsets of $A_k$. We note that the image of the $X_{ij}$'s under a $k$-algebra automorphism of $A_k$ is a special generating set. We use this fact below.

\begin{lem}\label{lem:coeff}$ $
\begin{enumerate}
\item Let $\phi$ be an $\cO_p$-algebra automorphism of $A_p$, $1\leq r\leq n$ and $1\leq s\leq m_r$. Then there exists some $1\leq r_0\leq n$ and $1\leq s_0\leq m_{r_0}$ such that $X_{r_0s_0}$ appears with unit coefficient in $\phi(X_{rs})$. If, in addition, $1\leq u\leq n$ and $1\leq v\leq m_u$ such that $X_{u_0v_0}$ appears with unit coefficient in $\phi(X_{uv})$, for some $1\leq u_0\leq n$ and $1\leq v_0\leq m_{u_0}$, then $q_{rs,uv}=q_{r_0s_0,u_0v_0}$. In particular, there does not exist $1\leq s_0'\leq m_{r_0}$ different from $s_0$ such that $X_{r_0s_0'}$ also appears with unit coefficient in $\phi(X_{rs})$. We have all the analogous results for $A_I$.
\item All $\cO_p$-algebra automorphisms of $A_p$ leave $T_p$ invariant.
\item Assume $p=2$ and $D_1$ is elementary abelian. All $\cO_I$-algebra automorphisms of $A_I$ leave $T_I$ invariant.
\end{enumerate}
\end{lem}

\begin{proof}$ $
\begin{enumerate}
\item We prove all the results for $A_p$ and $A_I$ simultaneously.
\newline
\newline
As noted in the proof of Lemma~\ref{lem:basic}, the $X_{ij}\in A_k$ are precisely the $X_i$'s constructed in the proof of~\cite[Corollary 4.3]{hk05}. Certainly $A_k$ is local, it is the basic algebra of a block with one simple module, and so $J(A_k)$ is the ideal generated by the $X_{ij}$'s. In particular, the $X_{ij}$'s form a basis of $J(A_k)/J^2(A_k)$. Now $\phi$ induces a $k$-algebra automorphism $\phi_k$ of $A_k$. Therefore, there exists some $X_{r_0s_0}$ appearing with non-zero coefficient in $\phi_k(X_{rs})$. Since an element $x\in\cO$ is invertible if and only if $x\notin \cJ$, the first claim follows.
\newline
\newline
Suppose $X_{r_0s_0}\in A_k$ (respectively $X_{u_0v_0}\in A_k$) appears with non-zero coefficient in $\phi_k(X_{rs})$ (respectively $\phi_k(X_{uv})$). We note that, since $p'$-roots of unity in $k$ lift uniquely to $\cO$, we need only prove that $\overline{q}_{rs,uv}=\overline{q}_{r_0s_0,u_0v_0}$. By~\cite[Lemma 2.3(i)]{bk07}, for any $X_{ij}$ that appears with non-zero coefficient in $\phi_k(X_{rs})$, $\overline{q}_{ij,u_0v_0}=\overline{q}_{r_0s_0,u_0v_0}$ and an analogous statement for any $X_{ij}$ appearing with non-zero coefficient in $\phi_k(X_{uv})$. In particular, if $(r,s)=(u,v)$, then
\begin{align*}
\overline{q}_{r_0s_0,u_0v_0}=\overline{q}_{u_0v_0,u_0v_0}=1=\overline{q}_{rs,uv}.
\end{align*}
Therefore, we assume $(r,s)\neq(u,v)$. In this case $X_{rs}X_{uv}\in J^2(A_k)\backslash J^3(A_k)$ and so there must exist some $X_{r_0's_0'}\in A_k$ (respectively $X_{u_0'v_0'}\in A_k$) appearing with coefficient $\lambda_{rs}\in k^\times$ (respectively $\lambda_{uv}\in k^\times$) in $\phi_k(X_{rs})$ (respectively $\phi_k(X_{uv})$) such that $X_{r_0's_0'}X_{u_0'v_0'}$ appears with non-zero coefficent in $\phi_k(X_{rs}X_{uv})$. If $X_{r_0's_0'}$ also appears with non-zero coefficient in $\phi_k(X_{uv})$, then
\begin{align*}
\overline{q}_{r_0s_0,u_0v_0}=\overline{q}_{r_0's_0',u_0v_0}=\overline{q}_{u_0v_0,r_0's_0'}^{-1}=\overline{q}_{r_0's_0',r_0's_0'}^{-1}=1
\end{align*}
and similarly $\overline{q}_{r_0's_0',u_0'v_0'}=1$. Then, by comparing the non-zero coefficient of $X_{r_0's_0'}X_{u_0'v_0'}=X_{u_0'v_0'}X_{r_0's_0'}$ in $\phi_k(X_{rs}X_{uv})=\overline{q}_{rs,uv}\phi_k(X_{uv}X_{rs})$, we get that $\overline{q}_{rs,uv}=1$. If $X_{r_0's_0'}$ does not appear with non-zero coefficient in $\phi_k(X_{uv})$, then $X_{r_0's_0'}X_{u_0'v_0'}$ appears with non-zero coefficients $\lambda_{rs}\lambda_{uv}$ in $\phi_k(X_{rs}X_{uv})$ and $\lambda_{rs}\lambda_{uv}\overline{q}_{r_0's_0',u_0'v_0'}^{-1}$ in $\phi_k(X_{uv}X_{rs})=\overline{q}_{rs,uv}^{-1}\phi_k(X_{rs}X_{uv})$. Therefore,
\begin{align*}
\overline{q}_{r_0s_0,u_0v_0}=\overline{q}_{r_0's_0',u_0v_0}=\overline{q}_{u_0v_0,r_0's_0'}^{-1}=\overline{q}_{u_0'v_0',r_0's_0'}^{-1}=\overline{q}_{r_0's_0',u_0'v_0'}=\overline{q}_{rs,uv}.
\end{align*}
For the final claim suppose such an $s_0'$ does exist and choose $1\leq u_0\leq t$ and $1\leq v_0\leq m_{u_0}$ such that $q_{r_0s_0,u_0v_0}\neq q_{r_0s_0',u_0v_0}$. The existence of $u_0$ and $v_0$ is guaranteed by part (3e) of Lemma~\ref{lem:basic}. Now, by the same reasoning as in the first paragraph, there exists some $X_{uv}$ such that $X_{u_0v_0}$ appears with unit coefficient in $\phi(X_{uv})$. By the second paragraph
\begin{align*}
q_{r_0s_0,u_0v_0}=q_{rs,uv}=q_{r_0s_0',u_0v_0},
\end{align*}
a contradiction.
\item Let $\phi$ be an $\cO_p$-algebra automorphism of $A_p$ and assume that $\phi(X_{rs})\notin T_p$, for some $1\leq r\leq t$ and $1\leq s\leq m_r$. By part (3d) of Lemma~\ref{lem:basic}, there exist $1\leq u\leq t$ and $1\leq v\leq m_u$ such that $q_{rs,uv}\neq 1$. Let $\mathbf{a}\in\mathbf{A}$ such that $X^{\mathbf{a}}$ appears with non-zero coefficient in either $\phi(X_{rs})$ or $\phi(X_{uv})$ and let $\mathbf{a}$ be minimal with respect to this property. We set these coefficients to be $a_{rs}$ and $a_{uv}$ respectively.
\newline
\newline
Let $X_{r_0s_0}$ (respectively $X_{u_0v_0}$) appear with with coefficient $a_{r_0s_0}\in\cO_p^\times$ (respectively $a_{u_0v_0}\in\cO_p^\times$) in $\phi(X_{rs})$ (respectively $\phi(X_{uv})$), note their existence is guaranteed by part (1). Furthermore let $X_{r_0s_0}$ appear with coefficient $b_{r_0s_0}$ in $\phi(X_{uv})$ and similarly $X_{u_0v_0}$ with coefficient $b_{u_0v_0}$ in $\phi(X_{rs})$.
\newline
\newline
$X^{\mathbf{a}}X_{r_0s_0}$ appears with coefficient $a_{uv}a_{r_0s_0}+a_{rs}b_{r_0s_0}$ in both $\phi(X_{uv}X_{rs})$ and $\phi(X_{rs}X_{uv})=q_{rs,uv}\phi(X_{uv}X_{rs})$. Now $1-q_{rs,uv}$ is invertible in $k$ and hence also in $\cO_p$ and so $a_{uv}a_{r_0s_0}+a_{rs}b_{r_0s_0}=0$. Similarly, by comparing coefficients of $X^{\mathbf{a}}X_{u_0v_0}$, we have that $a_{rs}a_{u_0v_0}+a_{uv}b_{u_0v_0}=0$. Taking these two equalities together gives $v_p(a_{rs})=v_p(a_{uv})$, where $v_p$ is the valuation of $\cO_p$ with respect to its unique maximal ideal. This implies $b_{r_0s_0}$ and $b_{u_0v_0}$ are both invertible. Part (1) of this lemma now gives $1=q_{r_0s_0,r_0s_0}=q_{rs,uv}$, a contradiction.
\item Let $\phi$ be an $\cO_I$-algebra automorphism of $A_I$ and assume that $\phi(X_{rs})\notin T_I$, for some $1\leq r\leq t$ and $1\leq s\leq m_r$. We define $u,v,\mathbf{a},a_{rs}$ and $a_{uv}$ exactly as in part (2). Without loss of generality, let $a_{rs}$ be non-zero. Note that by part (2), we must have $a_{rs},a_{uv}\in 2\cO_I$. As in part (2), there must exist some $X_{u_0v_0}$ (respectively $X_{r_0s_0}$) with unit coefficient in $\phi(X_{uv})$ (respectively $\phi(X_{rs})$). Let $X_{u_0v_0}$ appear with coefficient $a_{u_0v_0}$ in $\phi(X_{uv})$ and $b_{u_0v_0}$ in $\phi(X_{rs})$. We note that by part (1), $b_{u_0v_0}$ is not invertible.
\newline
\newline
We now study the coefficient of $X^{\mathbf{a}}X_{u_0v_0}$ in $\phi(X_{rs}X_{uv})$ and $\phi(X_{uv}X_{rs})$. By part (5) of Lemma~\ref{lem:basic} the only non-zero contributions must come from taking $X^{\mathbf{a}}$ in $\phi(X_{rs})$ and $X_{u_0v_0}$ in $\phi(X_{uv})$ or taking $X^{\mathbf{a}_1}X_{u_0(v_0-1)}$ with unit coefficient in $\phi(X_{rs})$ and $X^{\mathbf{a}_2}X_{u_0(v_0-1)}$ with unit coefficient in $\phi(X_{uv})$, where $\mathbf{a}_1+\mathbf{a}_2=\mathbf{a}$. (Note that as $b_{u_0v_0}$ is not invertible and $a_{uv}\in 2\cO_I$, $b_{u_0v_0}a_{uv}=0$ and so we need not consider taking $X_{u_0v_0}$ in $\phi(X_{rs})$ and $X^{\mathbf{a}}$ in $\phi(X_{uv})$.) In particular, the coefficients of $X^{\mathbf{a}}X_{u_0v_0}$ in $\phi(X_{rs}X_{uv})$ and $\phi(X_{uv}X_{rs})=q_{rs,uv}^{-1}\phi(X_{rs}X_{uv})$ are the same and, since $q_{rs,uv}\neq 1$ in $\cO_p$, zero. This implies the case of taking $X^{\mathbf{a}_1}X_{u_0(v_0-1)}$ with unit coefficient in $\phi(X_{rs})$ and $X^{\mathbf{a}_2}X_{u_0(v_0-1)}$ with unit coefficient in $\phi(X_{uv})$ must make a non-zero contribution in both $\phi(X_{uv}X_{rs})$ and $\phi(X_{rs}X_{uv})$.
\newline
\newline
Let $\mathbf{b}\in\mathbf{A}$ such that $X^{\mathbf{b}}X_{u_0(v_0-1)}$ appears with unit coefficient in $\phi(X_{rs})$ or $\phi(X_{uv})$ and let $\mathbf{b}$ be minimal with respect to this property. Note that $\mathbf{b}<\mathbf{a}$ since otherwise $X_{u_0(v_0-1)}$ itself appears with unit coefficient in either $\phi(X_{rs})$ or $\phi(X_{uv})$, contradicting the minimality of $\mathbf{b}$, unless $\mathbf{a}=\mathbf{b}=\varnothing$. In this case $X_{u_0(v_0-1)}$ appears with unit coefficient in $\phi(X_{rs})$ but this is a contradiction as, by part (1), $q_{rs,uv}=q_{u_0(v_0-1),u_0v_0}$ and, by part (3a) of Lemma~\ref{lem:basic}, $q_{u_0(v_0-1),u_0v_0}=1$.
\newline
\newline
Say $X^{\mathbf{b}}X_{u_0(v_0-1)}$ appears with unit coefficient in $\phi(X_{rs})$. Then we consider the coefficient of $X^{\mathbf{b}}X_{u_0(v_0-1)}X_{u_0v_0}$ in both $\phi(X_{rs}X_{uv})$ and $\phi(X_{uv}X_{rs})=q_{rs,uv}^{-1}\phi(X_{rs}X_{uv})$. In particular, we consider their images in $k$. The only non-zero contribution is from taking $X^{\mathbf{b}}X_{u_0(v_0-1)}$ in $\phi(X_{rs})$ and  $X_{u_0v_0}$ in $\phi(X_{uv})$. (Note that by the final part of (1), $X_{u_0(v_0-1)}$ cannot appear with unit coefficient in $\phi(X_{uv})$). So the coefficients are equal and non-zero. This is a contradiction as $\overline{q}_{rs,uv}\neq 1$.
\newline
\newline
If $X^{\mathbf{b}}X_{u_0(v_0-1)}$ appears with unit coefficient in $\phi(X_{uv})$ we consider the images in $k$ of the coefficients of $X^{\mathbf{b}}X_{u_0(v_0-1)}X_{r_0s_0}$ in $\phi(X_{rs}X_{uv})$ and $\phi(X_{uv}X_{rs})=q_{rs,uv}^{-1}\phi(X_{rs}X_{uv})$. The only non-zero contribution is from taking $X_{r_0s_0}$ in $\phi(X_{rs})$ and  $X^{\mathbf{b}}X_{u_0(v_0-1)}$ in $\phi(X_{uv})$. Comparing coefficients gives $\overline{q}_{rs,uv}=\overline{q}_{r_0s_0,u_0(v_0-1)}$ and so, by part (1), $\overline{q}_{r_0s_0,u_0v_0}=\overline{q}_{r_0s_0,u_0(v_0-1)}$. However, part (3c) of Lemma~\ref{lem:basic} now implies $\overline{q}_{r_0s_0,u_0v_0}=1$. As in part (1), this means $q_{r_0s_0,u_0v_0}=1$. Finally, by part (1), we have $q_{rs,uv}=q_{r_0s_0,u_0v_0}=1$, a contradiction.
\end{enumerate}
\end{proof}

\section{Weiss' condition and the main theorem}\label{sec:main}

In this section we prove our main result. Along with the results already proved in this article, our main tool will be an application of Weiss' condition. Weiss' condition is a statement about permutation modules originally stated in~\cite[Theorem 2]{we88} but proved in its most general form in~\cite[Theorem 1.2]{mcsyza18}. Proposition~\ref{prop:indMor} is a consequence of the condition that was proved in~\cite[Propositions 4.3,4.4]{eali19}. We first set up some notation.
\newline
\newline
Let $b$ be a block of $\cO H$, for some finite group $H$ and $Q$ a normal $p$-subgroup of $H$. We denote by $b^Q$ the direct sum of blocks of $\cO(H/Q)$ dominated by $b$, that is those blocks not annihilated by the image of $e_b$ under the natural $\cO$-algebra homomorphism $\cO H\to \cO(H/Q)$. In this section it is also necessary to extend our definition of $\cT(b)$ to include the possibility of $b$ being a direct sum of blocks.

\begin{prop}\label{prop:indMor}$ $
\begin{enumerate}
\item The inflation map $\Inf:\Irr(H/Q)\to\Irr(H)$ induces a bijection between $\Irr(b^Q)$ and $\Irr(b,1_Q)$.
\item Suppose $M$ is a $b$-$b$-bimodule inducing a Morita auto-equivalence of $b$ that permutes the elements of $\Irr(b,1_Q)$. Then ${}^QM$, the set of fixed points of $M$ under the left action of $Q$, induces a Morita auto-equivalence of $b^Q$. Furthermore, the permutation of $\Irr(b^Q)$ induced by ${}^QM$ is identical to the permutation that $M$ induces on $\Irr(b,1_Q)$, once these two sets have been identified using part (1).
\item If ${}^QM\in\cT(b^Q)$, then $M\in\cT(b)$.
\end{enumerate}
\end{prop}

\begin{lem}\label{lem:D1kernel}
Let $M\in\Pic(B)$ and $\sigma$ the corresponding permutation of $\Irr(B)$. Then $\sigma$ permutes the elements of $\Irr(B,1_{D_1})$.
\end{lem}

\begin{proof}
We first assume $p>2$ and that $B$ has a unique simple module. By part (1) of Lemma~\ref{lem:Op}, we need only check that $\sigma_p$, the corresponding permutation of
\begin{align*}
\{\cO_p\otimes_\cO V|V\text{ an }\cO\text{-free }B\text{-module affording some }\chi\in\Irr(B)\}
\end{align*}
induced by $\cO_p\otimes_\cO M$, permutes the $\cO_p\otimes_\cO V$ with $D_1$ acting trivially. Since $B$ has a unique simple module, $M$ is induced by an $\cO$-algebra automorphism of $B$, see~\cite[Proposition 2.2]{eali2019}. By part (4) of Lemma~\ref{lem:basic} and part (2) of Lemma~\ref{lem:coeff}, any $\cO_p$-algebra automorphism of $\cO_p\otimes_{\cO}B$ leaves $\cO_p\otimes_{\cO}((1-x)e_\varphi)_{x\in D_1}$ invariant, as desired.
\newline
\newline
We now drop the assumption that $B$ has a unique simple module. By Lemma~\ref{lem:ZEWeiss} we may apply part (2) of Proposition~\ref{prop:indMor} with respect to $B$, $M$ and $[D,Z(E)]$. Note that every character of $B^{[D,Z(E)]}$ reduces to some number of copies of the same irreducible Brauer character. In other words, $B^{[D,Z(E)]}$ is the direct sum of blocks each with a unique simple module.
\newline
\newline
By part (1) of Lemma~\ref{lem:tenchar}, any two of the blocks appearing in the direct sum $B^{[D,Z(E)]}$ are Morita equivalent via tensoring with an irreducible character of $\Irr(G,1_{D\times Z})$. Therefore, by the first paragraph, any Morita auto-equivalence of $B^{[D,Z(E)]}$ permutes $\Irr(B^{[D,Z(E)]},1_{D_1/[D,Z(E)]})$. The fact that $\sigma$ permutes $\Irr(B,1_{[D,Z(E)]})$ now follows from the last sentence in part (2) of Proposition~\ref{prop:indMor}.
\newline
\newline
In the above paragraph we need to be a little careful when we apply the conclusion from the first paragraph. First note that a block $C$ appearing in the direct sum $B^{[D,Z(E)]}$ may not be of the form of a block as described in $\S$\ref{sec:Bchar}. In particular, the relevant character $\varphi_C$ of $Z(E)$ may not be faithful. However, $C$ is naturally Morita equivalent to a block of $G/([D,Z(E)]\ker(\varphi_C))$ that will be of the desired form. Secondly we are implicitly using the fact that $[D/[D,Z(E)],E]=D_1/[D,Z(E)]$. This is required to ensure that ${}^{[D,Z(E)]}M$ does indeed permute $\Irr(B^{[D,Z(E)]},1_{D_1/[D,Z(E)]})$.
\newline
\newline
A slightly more delicate argument is required for the $p=2$ case due to the weaker result obtained for $p=2$ in Lemma~\ref{lem:Op}. When $B$ has a unique simple module we first use part (2a) of Lemma~\ref{lem:Op}, part (4) of Lemma~\ref{lem:basic} and part (2) of Lemma~\ref{lem:coeff} to apply part (2) of Proposition~\ref{prop:indMor} with respect to $B$, $M$ and $\Phi(D_1)$. In other words, we may assume that $D_1$ is elementary abelian. We can now use part (2b) of Lemma~\ref{lem:Op}, part (4) of Lemma~\ref{lem:basic} and part (3) of Lemma~\ref{lem:coeff} to apply part (2) of Proposition~\ref{prop:indMor} with respect to $B$, $M$ and $D_1$. The general $p=2$ case now follows exactly as for $p>2$.
\newline
\newline
This time we are implicitly using the fact that $[D/\Phi(D_1),E]=D_1/\Phi(D_1)$. This is required when we reduce to the situation of $D_1$ being elementary abelian.
\end{proof}

In what follows, for any finite group $H$ and $\alpha\in\Aut(H)$, we denote by $\Delta\alpha$ the subgroup $\{(g,\alpha(g))|g\in H\}$ of $H\times H$. If $\alpha=\Id_H$, we simply write $\Delta H$. $\cO_H$ (respectively $K_H$) will signify the trivial $\cO H$-module (respectively trivial $KH$-module). We are now ready to state and prove our main theorem.

\begin{thm}\label{thm:main}
Let $b$ be a block with normal abelian defect group and abelian inertial quotient. Then $\Pic(b)=\cL(b)$.
\end{thm}

\begin{proof}
We first note that $b$ is source algebra equivalent to a block of the form of $B$ as introduced in $\S$\ref{sec:Bchar}, where the defect group of $b$ is isomorphic to $D$ and its inertial quotient is isomorphic to $L$. The fact that we have a Morita equivalence follows from~\cite[Theorem A]{ku85} and that this Morita equivalence is in fact a source algebra equivalence from~\cite[Theorem 6.14.1]{li18b}. Note that in both of these articles actually an equivalence with a twisted group algebra $\cO_\alpha(D\rtimes L)$ is constructed. However, $\cO_\alpha(D\rtimes L)$ and $B$ are isomorphic as interior $D$-algebras, for an appropriately chosen $B$. (See the comments following~\cite[Theorem 4.2]{hk05} for a discussion of~\cite[Theorem A]{ku85}.) Note that $E$ must be a $p'$-group as otherwise $G$ has a normal $p$-subgroup strictly containing its defect group.
\newline
\newline
By~\cite[Lemma 2.8(ii)]{bkl18} the source algebra equivalence between $b$ and $B$ induces an isomorphism $\Pic(b)\cong\Pic(B)$ that restricts to an isomorphism $\cL(b)\cong\cL(B)$. Therefore, from now on we assume that $b=B$.
\newline
\newline
Let $M\in\Pic(B)$ and $\sigma$ the corresponding permutation of $\Irr(B)$. Let $\theta\in\Irr(D_2)$ be as defined in Lemma~\ref{lem:D2kernel} and $M_{\theta^{-1}}\in\Pic(B)$ the $B$-$B$-bimodule inducing the Morita auto-equivalence given by tensoring with $\theta^{-1}$. In other words,
\begin{align*}
M_{\theta^{-1}}=\cO(D_1\rtimes E)e_\varphi\otimes_\cO(\cO_{\theta^{-1}})\uparrow^{D_2\times D_2},
\end{align*}
where $\cO_{\theta^{-1}}$ is the $\cO(\Delta D_2)$-module affording the character $\theta$. In particular, $M_{\theta^{-1}}$ has linear source and, since $\cL(B)$ is a subgroup of $\Pic(B)$, we may replace $M$ with $M_{\theta^{-1}}\otimes_BM$. In other words, we may assume that $\sigma$ satisfies
\begin{align*}
\sigma(1_D,\chi)=\psi_\chi\otimes 1_{D_2},
\end{align*}
for all $\chi\in\Irr(E,\varphi)$, where $\psi_\chi\in\Irr(D_1\rtimes E,\varphi)$. However, by Proposition~\ref{lem:D1kernel}, we also have that $\psi_\chi\in\Irr(D_1\rtimes E,1_{D_1})$. Therefore, we can apply parts (2) and (3) of Proposition~\ref{prop:indMor} with respect to $B$, $M$ and $D$. Since $G/D$ is a $p'$-group, certainly ${}^DM\in\cT(B^D)$ and so $M\in\cT(B)\leq \cL(B)$ as required.
\end{proof}

In fact we can say more about $\Pic(B)$.

\begin{cor}\label{cor:main}
\begin{align*}
\Pic(B)&\cong\cT(\cO(D_1\rtimes E)e_\varphi)\times\cL(\cO D_2)\\
&=\cT(\cO(D_1\rtimes E)e_\varphi)\times(\Hom(D_2,\cO^\times)\rtimes\Aut(D_2)).
\end{align*}
\end{cor}

\begin{proof}
By the proof of Theorem~\ref{thm:main}, $\Pic(B)$ is generated by $\cT(B)$ and $\cL(\cO D_2)$. Now it is well-known that $\cL(P)=(\Hom(P,\cO^\times)\rtimes\Aut(P))$ for any finite $p$-group $P$. It therefore suffices to show that $\cT(B)\cong \cT(\cO(D_1\rtimes E)e_\varphi)\times\cT(\cO D_2)$.
\newline
\newline
Let $M\in\cT(B)$. It follows from~\cite[Theorem 1.1(ii), Remark 1.2(f)]{bkl18} that $M$ has vertex of the form $\Delta\alpha$ for some $\alpha\in N_{\Aut(D)}(L)$. In particular, $\alpha$ must respect the direct product $D=D_1\times D_2$. We claim that every $B$-$B$-bimodule summand of $\cO_{\Delta\alpha}\uparrow^{G\times G}$ that induces a Morita auto-equivalence of $B$ is of the form $M_1\otimes_\cO M_2$, for some $M_1\in\cT(\cO(D_1\rtimes E)e_\varphi)$ and $M_2\in\cT(\cO D_2)$. This is proved in~\cite[Lemma 2.3]{eali19} but with the added assumption that $\alpha$ is the identity. The argument in this more general setting is almost identical but we outline the main points for the convenience of the reader.
\newline
\newline
Every indecomposable direct summand of $\cO_{\Delta\alpha}\uparrow^{G\times G}$ is of the form $M_1\otimes_\cO M_2$, for $M_1$ an indecomposable summand of $\cO_{\Delta\alpha|_{D_1}}\uparrow^{(D_1\rtimes E)\times(D_1\rtimes E)}$ and $M_2$ an indecomposable direct summand of $\cO_{\Delta\alpha|_{D_2}}\uparrow^{D_2\times D_2}$. Now $K\otimes_\cO(M_1\otimes_\cO M_2)$ induces a bijection of simple $KB$-modules if and only if $K\otimes_\cO M_1$ (respectively $K\otimes_\cO M_2$) induces a bijection of simple $K(D_1\rtimes E)e_\varphi$-modules (respectively simple $KD_2$-modules). The claim now follows from~\cite[Th\'eor\`eme 1.2]{br90}.
\end{proof}

\begin{ack*}
The author would like to express gratitude to Prof. Burkhard K\"ulshammer for hosting the author at Friedrich-Schiller-Universit\"at Jena, where much of this article was written.
\end{ack*}

\end{document}